\documentclass[12pt]{article}
\usepackage{amsmath, amsthm}
\usepackage{amscd}
\usepackage{amssymb}
\usepackage{array}
\usepackage{color}

\newtheorem{thm}{Theorem}[section]
\newtheorem{theorem}[thm]{Theorem}

\newtheorem{lemma}[thm]{Lemma}
\newtheorem{proposition}[thm]{Proposition}

\theoremstyle{definition}
\newtheorem{definition}[thm]{Definition}
\newtheorem{example}[thm]{Example}

\newtheorem{observation}[thm]{Observation}

\newtheorem{remark}[thm]{Remark}

\begin{document}


\newcommand{\green}[1]{{\textbf{#1}}}
\newcommand{\Red}[1]{{\color{red}{#1}}}

\newcommand{\id}{\relax{\rm 1\kern-.28em 1}}
\newcommand{\R}{\mathbb{R}}
\newcommand{\C}{\mathbb{C}}
\newcommand{\Cbar}{\overline{\C}}
\newcommand{\Z}{\mathbb{Z}}
\newcommand{\Q}{\mathbb{Q}}
\newcommand{\kk}{k}
\newcommand{\bD}{\mathbb{D}}
\newcommand{\bG}{\mathbb{G}}
\newcommand{\bP}{\mathbb{P}}
\newcommand{\bM}{\mathbb{M}}
\newcommand{\g}{\mathfrak{G}}
\newcommand{\fh}{\mathfrak{h}}
\newcommand{\e}{\epsilon}
\newcommand{\Uu}{\mathfrak{U}}

\newcommand{\cA}{\mathcal{A}}
\newcommand{\cB}{\mathcal{B}}
\newcommand{\cC}{\mathcal{C}}
\newcommand{\cD}{\mathcal{D}}
\newcommand{\cI}{\mathcal{I}}
\newcommand{\cL}{\mathcal{L}}
\newcommand{\cK}{\mathcal{K}}
\newcommand{\cO}{\mathcal{O}}
\newcommand{\cG}{\mathcal{G}}
\newcommand{\cJ}{\mathcal{J}}
\newcommand{\cF}{\mathcal{F}}
\newcommand{\cP}{\mathcal{P}}
\newcommand{\cU}{\mathcal{U}}
\newcommand{\ep}{\mathcal{E}}
\newcommand{\E}{\mathcal{E}}
\newcommand{\cH}{\mathcal{O}}
\newcommand{\cV}{\mathcal{V}}
\newcommand{\cPO}{\mathcal{PO}}
\newcommand{\cHol}{\mathrm{H}}

\newcommand{\rGL}{\mathrm{GL}}
\newcommand{\rSU}{\mathrm{SU}}
\newcommand{\rSL}{\mathrm{SL}}
\newcommand{\rPSL}{\mathrm{PSL}}
\newcommand{\rSO}{\mathrm{SO}}
\newcommand{\rOSp}{\mathrm{OSp}}
\newcommand{\rSpin}{\mathrm{Spin}}
\newcommand{\rsl}{\mathrm{sl}}
\newcommand{\rsu}{\mathrm{su}}
\newcommand{\rM}{\mathrm{M}}
\newcommand{\rdiag}{\mathrm{diag}}
\newcommand{\rP}{\mathrm{P}}
\newcommand{\rdeg}{\mathrm{deg}}
\newcommand{\pt}{\mathrm{pt}}
\newcommand{\red}{\mathrm{red}}

\newcommand{\bm}{\mathbf{m}}

\newcommand{\M}{\mathrm{M}}
\newcommand{\End}{\mathrm{End}}
\newcommand{\Hom}{\mathrm{Hom}}
\newcommand{\diag}{\mathrm{diag}}
\newcommand{\rspan}{\mathrm{span}}
\newcommand{\rank}{\mathrm{rank}}
\newcommand{\Gr}{\mathrm{Gr}}
\newcommand{\ber}{\mathrm{Ber}}

\newcommand{\str}{\mathrm{str}}
\newcommand{\Sym}{\mathrm{Sym}}
\newcommand{\tr}{\mathrm{tr}}
\newcommand{\defi}{\mathrm{def}}
\newcommand{\Ber}{\mathrm{Ber}}
\newcommand{\spec}{\mathrm{Spec}}
\newcommand{\sschemes}{\mathrm{(sschemes)}}
\newcommand{\sschemeaff}{\mathrm{ {( {sschemes}_{\mathrm{aff}} )} }}
\newcommand{\rings}{\mathrm{(rings)}}
\newcommand{\Top}{\mathrm{Top}}
\newcommand{\sarf}{ \mathrm{ {( {salg}_{rf} )} }}
\newcommand{\arf}{\mathrm{ {( {alg}_{rf} )} }}
\newcommand{\odd}{\mathrm{odd}}
\newcommand{\alg}{\mathrm{(alg)}}
\newcommand{\sa}{\mathrm{(salg)}}
\newcommand{\sets}{\mathrm{(sets)}}
\newcommand{\smflds}{\mathrm{(smflds)}}
\newcommand{\mflds}{\mathrm{(mflds)}}
\newcommand{\shcps}{\mathrm{(shcps)}}
\newcommand{\sgrps}{\mathrm{(sgrps)}}
\newcommand{\SA}{\mathrm{(salg)}}
\newcommand{\salg}{\mathrm{(salg)}}
\newcommand{\varaff}{ \mathrm{ {( {var}_{\mathrm{aff}} )} } }
\newcommand{\svaraff}{\mathrm{ {( {svar}_{\mathrm{aff}} )}  }}
\newcommand{\ad}{\mathrm{ad}}
\newcommand{\Ad}{\mathrm{Ad}}
\newcommand{\pol}{\mathrm{Pol}}
\newcommand{\Lie}{\mathrm{Lie}}
\newcommand{\Proj}{\mathrm{Proj}}
\newcommand{\rGr}{\mathrm{Gr}}
\newcommand{\rFl}{\mathrm{Fl}}
\newcommand{\rPol}{\mathrm{Pol}}
\newcommand{\rdef}{\mathrm{def}}
\newcommand{\ah}{\mathrm{ah}}

\newcommand{\uspec}{\underline{\mathrm{Spec}}}
\newcommand{\uproj}{\mathrm{\underline{Proj}}}

\newcommand{\sym}{\cong}

\newcommand{\al}{\alpha}
\newcommand{\lam}{\lambda}
\newcommand{\de}{\delta}
\newcommand{\ttau}{\tilde \tau}
\newcommand{\D}{\Delta}
\newcommand{\s}{\sigma}
\newcommand{\lra}{\longrightarrow}
\newcommand{\ga}{\gamma}
\newcommand{\ra}{\rightarrow}

\newcommand{\wbar}{\overline{w}}
\newcommand{\zbar}{\overline{z}}
\newcommand{\fbar}{\overline{f}}
\newcommand{\etabar}{\overline{\eta}}
\newcommand{\zetabar}{\overline{\zeta}}
\newcommand{\betabar}{\overline{\beta}}
\newcommand{\albar}{\overline{\alpha}}
\newcommand{\abar}{\overline{a}}
\newcommand{\tbar}{\overline{t}}
\newcommand{\thetabar}{\overline{\theta}}
\newcommand{\Mbar}{\overline{M}}
\newcommand{\aubar}{\underline{a}}
\newcommand{\fg}{\mathfrak{g}}
\newcommand{\Span}{\mathrm{span}}

\newcommand{\NOTE}{\bigskip\hrule\medskip}

\newcommand{\G}{{(\C^{1|1})}^\times}
\newcommand{\pair}[2]{\langle \, #1, #2\, \rangle}
\newcommand{\cinfty}{\mathcal{C}^\infty}

\medskip

\centerline{\Large \bf   SUSY structures, representations and\\} 

\medskip

\centerline{\Large \bf  Peter-Weyl theorem for $S^{1|1}$}

\bigskip

\centerline{C. Carmeli$^\natural$, R. Fioresi$^\flat$, S. Kwok$^\flat$}

\medskip

\centerline{\it $^\natural$ DIME, 
Universit\`a di Genova, Genova, Italy}
\centerline{{\footnotesize e-mail: carmeli@dime.unige.it}}

\medskip

\centerline{\it $^\flat$ Dipartimento di Matematica, Universit\`{a} di
Bologna }
 \centerline{\it Piazza di Porta S. Donato, 5. 40126 Bologna. Italy.}
\centerline{{\footnotesize e-mail: rita.fioresi@UniBo.it,
stephendiwen.kwok@unibo.it}}

\begin{abstract} 
The real compact supergroup $S^{1|1}$ is analized from different perspectives and its representation theory is studied. We 
prove it is the only (up to isomorphism) supergroup, which is a
real form of $(\C^{1|1})^\times$ with reduced Lie group $S^1$, and a link with SUSY structures on $\C^{1|1}$ is established.
We  describe a large family of complex semisimple
representations of $S^{1|1}$  and we show that any $S^{1|1}$-representation whose weights are all nonzero is a direct sum of members of our family. 
We also compute the matrix elements of the members of this family and 
we give a proof of the Peter-Weyl theorem for  $S^{1|1}$.
\end{abstract}

\section{Introduction} 
\label{intro}

In this paper we discuss the real compact supergroup $S^{1|1}$ and its complex representations from various perspectives. We will show that already in 
this case,  the representation theory of compact supergroups  is more subtle than its ordinary counterpart and we show how the Peter Weyl theorem has to be 
suitably modified.

\medskip

We start with an introduction to real forms of complex supermanifolds
taking into account also the functor of points point of
view. We then discuss the real forms of supergroups using the
equivalent language of Super Harish-Chandra Pairs (SHCP's)
and we construct all of the real forms of the supergroup
$(\C^{1|1})^\times$ corresponding, in the ordinary setting,
to the reduced group $S^1$. We first construct such real forms, via
the SHCP's approach and then we recover the same real forms via
a purely geometric approach, that is,
via the functor of points. We also prove that, up to isomorphism, there
is only one of such real forms and we call it $S^{1|1}$, since its
reduced group is $S^1$ and it is a real Lie supergroup of dimension
$1|1$.

\smallskip
In the geometric
approach, we furtherly see that all of the involutions giving rise to a real
form of $(\C^{1|1})^\times$ come as the composition of a conjugation, whose
reduced form is the usual complex conjugation, and the two SUSY preserving
automorphisms $P_\pm$ (see Sec. \ref{geo-S11}). 
We prove in fact that $P_\pm$ are the
only holomorphic automorphisms of $\C^{1|1}$, which preserve its
SUSY structure, as Manin defines it in \cite{ma1}. This 
connection between the real forms of $(\C^{1|1})^\times$ and the
SUSY preserving automorphims is
surprising, but only apparently. In fact the supergroup structure of 
$(\C^{1|1})^\times$ is best understood as 
modelled after the superalgebra $\bD$ (refer to the work by 
Bernstein \cite{dm}),
and furthermore SUSY structures can also be interpreted in the
framework of $\bD$. We hence believe that our paper helps to shed light to 
some aspects of SUSY curves, like their interpretation 
through the superalgebra $\bD$,
which have not yet been fully developed and understood.

\medskip

In the second part of the paper, in Sec. \ref{S11-reps} 
we describe 
the complex representations on $S^{1|1}$ and we give
a concrete and constructive proof of the Peter-Weyl theorem in this
context; then in Sec. \ref{shcp-rep-fns}
we go to a more abstract approach, via the SHCP's.

\medskip

{\bf Acknoledgements}. 
C.C. and S.K. wish to thank the Bologna Department of Mathematics for
its hospitality while this work was completed. The authors also wish to
thank prof. V. S. Varadarajan and prof. L. Migliorini
for many illuminating discussions.
A special thank goes to Prof. P. Deligne for a private communication
on the construction of real structures of supermanifolds.

\section{Preliminaries} \label{prelim}

We briefly summarize few definitions and key facts about
supergeometry to establish our notation, for all the details
see \cite{ccf} and \cite{dm}.

\medskip
\noindent
Let us take our ground field to be either the complex or
the real field. 

\medskip

A {\it superalgebra} $A$ is a $\Z_2$-graded algebra, $A=A_0 \oplus A_1$,
where  $p(x)$ denotes the parity of a homogeneous element $x$.
The superalgebra $A$ is said to be {\it commutative} if for any two
homogeneous elements $x, y$, $xy=(-1)^{p(x)p(y)}yx$.
All superalgebras are assumed to be commutative unless otherwise specified.

\begin{definition}
A {\it superspace} $S=(|S|, \cO_S)$ is a topological space $|S|$
endowed with a sheaf of superalgebras $\cO_S$ such that the stalk at
a point  $x\in |S|$ denoted by $\cO_{S,x}$ is a local superalgebra.
A {\it morphism} $\phi:S \lra T$ of superspaces is given by
$\phi=(|\phi|, \phi^*)$, where $\phi: |S| \lra |T|$ is a map of
topological spaces and $\phi^*:\cO_T \lra \phi_*\cO_S$ is
a local sheaf morphism.
A  \textit{differentiable (resp. analytic) supermanifold}
of dimension $p|q$ is a super ringed space $M=(|M|, \cO_M)$
where $|M|$ is a second countable, Hausdorff topological space, and $\cO_M$ is a sheaf of superalgebras over $\R$ (resp. $\C$),
which is locally isomorphic  to 
$\R^{p|q}$ (resp. $\C^{p|q}$). This means that  for each point $x \in |M|$ there exists
an open neighbourhood $U_x \subset |M|$ such that:
$$
\cO_M|_{U_x} \cong C^\infty_{\R^p}|_{U_x'} \otimes \wedge_\R(\xi_1, \dots, \xi_q)
$$
(resp. $
\cO_M|_{U_x} \cong {\mathcal H}_{\C^p}|_{U_x'} 
\otimes \wedge_\C(\xi_1, \dots, \xi_q)
$)
where $C^\infty_{\R^p}$ (resp. ${\mathcal H_{\C^p}}$) is
the ordinary sheaf of $C^\infty$  functions over $\R^p$ (resp.
holomorphic functions over $\C^p$), and
$\wedge_\R(\xi_1, \dots, \xi_q)$ (resp. $\wedge_\C(\xi_1, \dots, \xi_q)$) denotes the real (resp. complex) exterior algebra in $q$ variables
and $U_x'$ an open subset of $\R^p$ (resp. $\C^p$).
\end{definition}

\medskip

The given definition works also, with suitable changes, if one wants
to define real analytic supermanifolds, however we shall be mostly
interested in the real differentiable or the complex analytic category.
For the moment our definitions are general enough to work in any
of these three very different categories, hence we shall say 
``supermanifold'' without further specifications, whenever
our results or definitions do not depend on a specific one of the three
categories. 

\medskip
Next, we introduce the notion of $T$-point and functor of points
of a supermanifold.

\begin{definition} \label{Tpt}
\index{$T$-point}
Let $M$ and $T$ be supermanifolds.  
A \textit{$T$-point} of $M$ is 
a morphism $T \longrightarrow M$.  We denote the set of all $T$-points 
by $M(T)$. We define the \textit{functor of points} of the supermanifold $M$ 
the functor:
$$
M: \text{(smflds)}^o \lra \sets, \quad T \mapsto S(T), \quad
S(\phi)(f)=f \circ \phi,
$$
where $\text{(smflds)}$ denotes the category of supermanifolds and
the index $o$ as usual refers to the opposite category. We
shall write $\text{(smflds)}_\R$ or $\text{(smflds)}_\C$ whenever it is
necessary to distinguish between real or complex supermanifolds.
\end{definition} 

By Yoneda lemma, a supermanifold $M$ can be studied through its 
functor of points (see \cite{ccf}, for more details).

We now define the real supermanifold underlying a complex supermanifold
following \cite{dm}.

\begin{definition} \label{complexconj}
Let $M=(|M|, \cO_M)$ be a complex super manifold. We define a
\textit{complex conjugate} of $M$ as a complex super manifold
$\bM=(|M|, \cO_{\bM})$, where now $\cO_\bM$ is just a
supersheaf, together with a ringed space $\C$-antilinear
isomorphism  $\cO_M \cong  \cO_{\bM}$. 
If we choose the supersheaf of the complex
conjugate to be $\cO_M$ with the $\C$-antilinear structure and
the ringed space $\C$-antilinear
isomorphism to be the identity, we call
such complex conjugate  $\Mbar=(|M|, \cO_{\Mbar})$.
For convenience we shall denote the map realizing the
$\C$-antilinear isomorphism between $M$ and $\Mbar$ as
$\sigma: M \lra \Mbar$ and sometimes we shall write $(\sigma^*)^{-1}(f)=\fbar$
(though strictly speaking $\fbar=f$).
\end{definition}

\begin{remark} \label{antiholo}
It is important to notice that when $M$ is an ordinary complex
manifold, the sheaf $\cO_{\Mbar}$ we just defined  is not the 
sheaf of antiholomorphic functions on $M$, as one may expect, but it is 
isomorphic to it.
If $\cO_M^{\ah}$ is the sheaf of the antiholomorphic functions on
$M$, that is for a suitable cover, we have:
$$ 
\cO_M^{\ah}(U)=\{\hat{f} \, | \, f \in \cO_M(U)\}, \qquad
\hat{f}(z):=\overline{f(z)}
$$
($\lambda \cdot \hat{f}=\hat{\lambda} \hat{f}$).
Then we have the  isomorphism: $\phi:\cO_M^{\ah} \lra \cO_{\Mbar}$,
$\phi(\hat{f})=f$. So $M^{\ah}:=(|M|, \cO^{\ah})$ is a complex conjugate
of $M$, but $\Mbar\neq M^\ah$.
\end{remark} 

\medskip
For convenience, from now on, we shall choose the complex conjugate
to be $\Mbar$, though the reader must be aware that this is one
of the many possible choices and furthermore keep in mind 
the relation between the sheaf $\cO_{\Mbar}$ and the sheaf of 
antiholomorphic functions in the ordinary case.

\begin{definition} \label{realforms} 
We define a \textit{real
structure} on $M$ as an involutive isomorphism of ringed spaces
$\rho: M \lra  \Mbar$, which is $\C$-linear on the sheaves, that is
$|\rho|: |M| \lra |M|$ is involutive, i.e. $|\rho|^2={\mathrm{id}}$, and
$\rho^*: \cO_{\Mbar} \lra \rho_*\cO_M$ is a $\C$-linear sheaf isomorphism.
We furtherly define the isomorphism of ringed superspaces
$\psi=\sigma^{-1} \circ \rho: M \lra M$, which is
$\C$-antilinear on the sheaves $\psi^*=\rho^*\circ (\sigma^*)^{-1}:
\cO_M \lra  \rho_*\cO_M$.
\end{definition}

Once  a real structure $\rho\colon M\to \Mbar$ is given, 
one defines the topological space $|M|^{|\rho|}$ consisting of 
the fixed points of $\rho:|M| \lra |M|$. 
Hence it is  possible to consider the restriction
\(
 \cO_M|_{M^\rho}
\)
and define the superspace $(|M^\rho|, \cO_M|_{M^\rho})$.  
The morphism $\psi$ restricts to a morphism (still denoted by $\psi$) 
\[ 
\psi \colon (|M^\rho|, \cO_M|_{M^\rho})\to (|M^\rho|, \cO_M|_{M^\rho})
\] 
whose reduced part is the identity.
Hence it is meaningful to define  the set  $\cO_{M^\rho}$ of sections such that
\(\psi^\ast(f)=f\), for each $f\in  \cO_M|_{M^\rho}$. We say that the
real supermanifold $M^\rho=(|M|^{|\rho|}, \cO_{M^\rho})$ 
is the  \textit{real form} of $M$ defined by $\rho$.

\medskip

We now want to take into account the real forms of a complex
analytic supergroup. Let $G$ be a complex supergroup, $\overline{G}$ 
inherits naturally a supergroup structure. 

\begin{definition}
We say that a real structure $\rho$ on $G$ is a \textit{real supergroup
structure} if $\rho$ is a supergroup morphism. 
As one can readily check the fact that $\rho$ is a supergroup
morphism guarantees that $M^\rho$ is indeed a (real) supergroup; in fact
we have on $\cO_M$ the comultiplication, counit and antipode morphisms
which suitably restrict to $\cO_M|_{M^\rho}$. 
\end{definition}

Through the notion real form it is possible to define
the concept of real underlying supermanifold, which is mostly important for us.

\begin{definition} \label{realforms1} 
Define on $M \times \Mbar$ the real
structure $\tau:  M \times \Mbar \lra \Mbar \times M$, where
$$
|\tau|: |M| \times |M| \lra |M| \times |M|, \qquad 
|\tau|(x,y)=(y,x)
$$
and $\tau^*: \cO_{\Mbar \times M} \lra \tau_*\cO_{M \times \Mbar}$ is
defined
(on products of open sets) as:
$$
\tau(f \otimes g)=g \otimes f
\qquad f \in \cO_{\Mbar}, \, g \in \cO_M.
$$
Notice that it is
enough to specify the image of $f \otimes g \in \cO_{\Mbar \times M}$
in order to have $\tau^*$ defined everywhere (see Ch. 4 \cite{ccf}).
If we choose  local coordinates $(z_i, \theta_j)$ of $M$ and 
$(w_k, \eta_l)$ of $\Mbar$, 
belonging to a suitable open cover of $|M|$, we have:
$$
\tau^*(w_i \otimes 1)=1 \otimes w_i, \, 
\tau^*(\eta_j \otimes 1) = 1 \otimes \eta_j, \, 
\tau^*(1 \otimes z_i)=z_i \otimes 1, \,
\tau^*(1 \otimes \theta_j)=\theta_j \otimes 1 
$$

On the fixed points of $|\tau|$, that is on the diagonal
$\Delta \subset |M| \times |M|$, the sheaf of the
real supermanifold $M^\tau$ is
locally given by the sections which are invariant under 
$\psi^*=\tau^* \circ ((\sigma^*)^{-1} \otimes \sigma^*)$
$: \cO_{M \times \Mbar} \lra \tau_* \cO_{M \times \Mbar}$, namely
$$
\begin{array}{cc}
(z_i\otimes 1 + 1\otimes (\sigma^*)^{-1}(z_i))/2, & (\theta_j \otimes 1 + 
1 \otimes (\sigma^*)^{-1}(\theta_j))/2, \\ \\
(z_i\otimes 1 - 1 \otimes (\sigma^*)^{-1}(z_i))/2i, & 
(\theta_j \otimes 1- 1\otimes (\sigma^*)^{-1}(\theta_j))/2i
\end{array}
$$
and
$$
\begin{array}{cc}
(1\otimes w_i + \sigma^*(w_i)\otimes 1)/2, & (1 \otimes \eta_j  + 
1 \otimes (\sigma^*)^{-1}(\eta_j))/2, \\ \\
(1\otimes w_i - \sigma^*(w_i)\otimes 1)/2i, & (1 \otimes \eta_j  - 
1 \otimes (\sigma^*)^{-1}(\eta_j))/2i.
\end{array}
$$
Notice that the second set of invariant sections is not really a new
one, since we may always choose $w_i=(\sigma^*)^{-1}(z_i)$ and
$\eta_j=(\sigma^*)^{-1}(\eta_j)$, thus retrieving the previous set
(in any case, even if we do not make the choice $w_i=(\sigma^*)^{-1}(z_i)$, 
$w_i$ will be the image under $(\sigma^*)^{-1}$ of some element in
$\cO_M$ and similarly for $\eta_j$).

\medskip
Notice that here the role of the real structure $\rho$ (see Def. 
\ref{realforms})
is played by $\tau$ and the role of $(\sigma^*)^{-1}$ by 
$(\sigma^*)^{-1} \otimes \sigma^*$, since $M^\tau$ is a real
form of $M \times \Mbar$ corresponding to the real structure $\tau$.

\medskip
As we remarked, it is customary to denote $\zbar_i:=(\sigma^*)^{-1}(z_i)$
and  $\thetabar_j:=(\sigma^*)^{-1}(\theta_j)$, to
forget the tensor product and furthermore to write the
real local coordinates of $M^\tau$  as:
$$
x_i= (z_i + \zbar_i)/2, \quad y_i = (z_i - \zbar_i)/2i, \qquad
\mu_j=   (\theta_j + \thetabar_j)/2,\quad
\nu_j=   (\theta_j - \thetabar_j)/2i.
$$
We shall call $M^\tau$ the \textit{real underlying supermanifold}
and we shall denote it with $M_\R$.
\end{definition}

We now turn to examine an example of particular importance to us.

\begin{example}\label{realform-ex11}
We want to construct $\C^{1|1}_\R$, 
the real supermanifold underlying $\C^{1|1}$. 
Let us denote $\Cbar^{1|1}$ the complex
conjugate introduced above. We define the real structure $\tau$ on 
$\C^{1|1} \times \Cbar^{1|1}$ as follows. 
On the topological space we define $|\tau|: |\C| \times |\C| \lra
|\C| \times |\C|$, $|\tau|(p,q)=(q,p)$ and on the sheaves as 
the $\C$-linear isomorphism 
$\tau^*: \cO_{\Cbar^{1|1} \times \C^{1|1}} \lra \tau_* \cO_{\C^{1|1} \times
\Cbar^{1|1}}$,
$\tau^*(w,\eta,z,\zeta)=(z, \zeta, w, \eta)$,
where $(z,\zeta)$ and $(w, \eta)$ global coordinates
on $\C^{1|1}$ and $\Cbar^{1|1}$ respectively. We associate to $\tau^*$
the $\C$-antilinear isomorphism
$\psi^*=\tau^* \circ ((\sigma^*)^{-1} \otimes \sigma^*)$:
$$
\psi^*(w,\eta,z,\zeta)=(\zbar, \zetabar, \wbar, \etabar)
$$
where the meaning of $\psi^*$ is understood with the above
conventions.

We warn the reader that the $\, \bar{} \,$ in our definitions
denotes both $\sigma^*$ and
$(\sigma^*)^{-1}$ as it is customary (refer to Remark \ref{antiholo}
to relate it to the ordinary setting).

So the global coordinates on the real supermanifold $\C^{1|1}_\R$ are the
elements 
inside $\cO_{\C^{1|1} \times
\Cbar^{1|1}}|_{\Delta}$ ($\Delta$ denoting
the diagonal in $|\C^{1|1}| \times
|\C^{1|1}|$): 
\begin{equation}\label{realcoordeq}
x=(z+\zbar)/2, \quad y=(z-\zbar)/2i, \quad \mu=(\zeta+\zetabar)/2, \quad
\nu=(\zeta-\zetabar)/2i
\end{equation}
which are evidently invariant under $\psi^*$.
We can also think of $z$, $\zbar$, $\zeta$, $\zetabar$ as the ``generators''
(in a topological sense)
of $\cO_{\C^{1|1}} \otimes \C$ inside $\cO_{\C^{1|1} \times
\Cbar^{1|1}}|_{\Delta}$. Sometimes  $z$, $\zbar$, $\zeta$, $\zetabar$ are
improperly called ``coordinates'' of ${\C^{1|1}_\R}$, since one
can readily recover from them the coordinates of  ${\C^{1|1}_\R}$ via
the formula (\ref{realcoordeq}).

\medskip
We now turn to the problem of examining the functor of points of $\C^{1|1}_\R$. 
Since we have global coordinates $x$, $y$, $\mu$, $\nu$ we have that:
$$
\begin{array}{rl}
\C^{1|1}_\R(T)&=\Hom_{\smflds_\R}(T, \C^{1|1}_\R)=\Hom_{\salg_\R}
 (\cO(\C^{1|1}_\R), \cO(T))= \\ \\
&=\{ \phi:  \cO(\C^{1|1}_\R)
\lra \cO(T)\}= \\ \\
&=\{(t_0,t_1, \theta_0, \theta_1) \, | \, 
t_i \in \cO(T)_0, \, \theta_i \in \cO(T)_1\}
\end{array}
$$
because we specify the morphism $\phi$ by giving the image of 
the four real sections $x$, $y$, $\mu$, $\nu$ ($(\salg)_\R$
denotes the category of real commutative superalgebras).
Notice that giving the real morphism $\phi$ is
equivalent to give the complex morphism:
$$
\phi':  \cO(\C^{1|1}) \lra \cO(T)\otimes \C
$$
In fact $\phi'(z)=t_0+it_1$, $\phi'(\zeta)=\theta_0+i\theta_1$, thus
retrieving the four real elements $t_0$, $t_1$, $\theta_0$, $\theta_1$
which are the images of the four real generators detailed above. 
This shows that given $\phi'$ we can retrieve $\phi$, but of course
the other way around is clear too. 
Hence we can write
equivalently 
$$
\C^{1|1}_\R(T)=\Hom_{\salg_\C}(\cO(\C^{1|1}),\cO(T)\otimes \C), \quad
T \in \smflds_\R
$$
in accordance with the definition given in \cite{fi}.

\medskip
One may also describe the functor of points of $\C^{1|1}_\R$ through
the coordinates $z$, $\zbar$, $\zeta$, $\zetabar$, in other words
one looks at the morphisms:
$$
\alpha:  \cO(\C^{1|1}_\R)\otimes \C 
\lra \cO(T)\otimes \C.
$$
specified once we know the images of
$z$, $\zbar$, $\zeta$, $\zetabar$, with $\al(\zbar)=\overline{\al(z)}$
and $\al(\zetabar)=\overline{\al(\zeta)}$. 
Hence $\alpha$ is identified
with the quadruple $(\alpha(z)=t, \alpha(\zbar)=\overline{t},
\alpha(\zeta)=\theta, \alpha(\zetabar)=\thetabar)$, where
$t=t_1+it_2$, $\overline{t}=t_1-it_2$, $\theta=\theta_1+i\theta_2$
and $\thetabar=\theta_0-i\theta_1$. So again a morphism $\alpha$ is
identified with the quadruple $(t_1,t_2, \theta_1, \theta_2)$
as above.

\end{example}

Given a complex supermanifold $M$
the real form $M^\rho$ associated with a real structure $\rho$
may be realized as a submanifold of $M_\R$ as follows.

\medskip
Let $\Gamma_\rho$ be the graph of $\rho$. $\Gamma_\rho$ is
a complex analytic subsupermanifold of $M \times \Mbar$
with underlying topological space $|\Gamma_\rho|=(p, |\rho|(p))$,
$p \in |M|$ and sheaf:
$$
\cO_{\Gamma_\rho}=\cO_{M \times \Mbar}/\cI
$$
where $\cI$ is the ideal generated by the elements
$1 \otimes f - \rho^*(f) \otimes 1$.
$\Gamma_\rho$ is isomorphic to $M$, as one can readily see in the
language of the functor of points:
$$
\begin{array}{cccc}
\phi_T: & M(T) & \lra & \Gamma_\rho(T) \\
 & t & \mapsto & (t, \rho_T(t))
\end{array}
$$
We now consider the commutative diagram:
$$
\begin{CD}
|M|^{|\rho|} @> >> \Delta \cap |\Gamma_\rho| \\
@VV V @VV V \\
(|M|, \cO_M) @> >> (\Gamma_\rho, \cO_{\Gamma_\rho})
\end{CD}
$$  
where the horizontal arrows are isomorphisms, while the
vertical arrows mean the inclusion of the topological
space into the corresponding supermanifold.
Let $M^\rho_\C=(|M|^{|\rho|}, \cO_{M^\rho} \otimes \C)$.
We have that:
$$
\cO_{M^\rho} \otimes \C= \cO_M|_{|M|^{|\rho|}}
$$

Hence from the commutative diagram above, we have that:
$$
\cO_{\Gamma_\rho}|_{\Delta \cap |\Gamma_\rho|} = 
(\cO_{M \times \Mbar}/\cI)|_{\Delta \cap |\Gamma_\rho|}=
\cO_{M^\rho} \otimes \C
$$
from which we can easily retrieve the sheaf $\cO_{M^\rho}$. Locally
it will be given by equations in the 
coordinates $z_i$, $\zbar_i$, $\zeta_j$, $\zetabar_j$
we use to describe $\cO_{M_\R} \otimes \C$.

\medskip
Let us look at an example to elucidate our discussion
in a special case, which is of particular interest to us.

\begin{example} \label{realform-fpts}
Let us consider $\C^{1|1}_\R$ discussed in Example \ref{realform-ex11}
and consider the real structure $\rho: \C^{1|1} \lra \Cbar^{1|1}$
given by $|\rho|(p)=\overline{p}$, $p \in |\C|$ and
$\rho^*(z)=\zbar$, $\rho^*(\zeta)=\zetabar$, once the global
coordinates are chosen as in \ref{realform-ex11}.
Clearly $|\C|^{\rho}=|\R|$ and the real form $(\C^{1|1})^\rho$
is retrieved as the real subsupermanifold of $\C^{1|1}_\R$
obtained by taking the quotient of $\cO_{\C^{1|1}_\R}$ by the
ideal sheaf locally generated by:
$$
z-\rho^*(z), \qquad \zeta-\rho^*(\zeta)
$$
or equivalently by: $z-\zbar$, $\zeta-\zetabar$, where we need to
reexpress the sections $z$, $\zbar$,  $\zeta$, $\zetabar$ in terms
of the real coordinates $x$, $y$, $\mu$, $\nu$. The ideal
sheaf is then generated by the elements $y$ and $\nu$, hence
we retrieve the supermanifold $\R^{1|1}$ as one expects.

\medskip
In terms of the functor of points we have that $(\C^{1|1})^{\rho}(T)$
consists of $(t,s,\xi, \eta) \in (\C^{1|1})(T)$, that is $t$, $s$ in $\cO(T)_0$
and $\xi$, $\eta$ in $\cO(T)_1$ (refer to Example \ref{realform-ex11}) with 
$s=\eta=0$, hence $(\C^{1|1})^{\rho}(T)$ is naturally identified
with $\R^{1|1}(T)$. 
Equivalently we can write (see Example \ref{realform-ex11}) as
the elements in $(\C^{1|1})(T)$ satisfying some relations:
\begin{equation}\label{foptsR11}
\R^{1|1}(T)=\{(t,\tbar, \theta, \thetabar) \in \C^{1|1}(T) \, | \,
t=\tbar, \, \theta=\thetabar \}
\end{equation}
where the $t$, $\tbar$, $\theta$, $\thetabar \in \cO(T)\otimes \C$ 
are the images of the
``coordinates'' $z$, $\zbar$,  $\zeta$, $\zetabar$ in the sense
expressed in \ref{realform-ex11}. We can retrieve the $T$-points
$\R^{1|1}(T)=\{t_1,\theta_1)\}$ as $\{(t_1, 0, \theta_1, 0)\}$ simply
by expressing $t$, $\tbar$, $\theta$, $\thetabar$ in (\ref{foptsR11}) 
in terms of the
real and imaginary parts. 
\end{example}

\section{Super Harish-Chandra Pairs and Real Forms}
\label{shcp-sec}

A \textit{Lie supergroup} is group object in the category of
supermanifolds. This is equivalent to ask that the functor of points
is group valued. A Lie supergroup is \textit{compact} if 
its underlying topological space is compact.

\medskip

A very effective approach to the theory of Lie supergroups is via
the SHCP's. We are going to briefly recall the definition and main
property sending the reader to \cite{ccf} Ch. 7 for all of the details.

\begin{definition} \label{def:SHCP}
Suppose $(G_0,\fg)$ are respectively a group (real Lie or 
complex analytic) and a super
Lie algebra. Assume that:
\begin{enumerate}
\item $\fg_0 \simeq {\rm Lie}(G_0)$ (here $\simeq$ denotes real or complex linear isomorphism depending on the category we are considering),

\item  $G_0$ acts on $\fg$ and this action restricted to
$\fg_0$ is the adjoint representation of $G_0$ on
    $\Lie(G_0)$. Morever the differential of such action is
the Lie bracket. We shall denote such an action with $\Ad$ or
as $g.X$, $g \in G_0$, $X \in \fg$. 
\end{enumerate}
Then $(G_0,\fg)$ is called a \emph{super Harish-Chandra pair
(SHCP)}.

\smallskip\noindent
A \textit{morphism} of SHCP is simply a pair of morphisms 
$\psi=( \psi_0,\rho^\psi )$ preserving the SHCP structure
that is:
\begin{enumerate}
\item $\psi_0 : G_0 \rightarrow H_0$ is a Lie group morphism (in the
analytic or differential category);
\item $\rho^\psi:\fg \rightarrow \fh$ is a super Lie algebra morphism 
(real or complex linear morphism depending on the category we 
are considering),

\item $\psi_0$ and $\rho^\psi$ are compatible in the sense that:
\begin{eqnarray*}
\rho^\psi_{\left.\right| \fg_0} & \simeq & 
d\psi_0\qquad \Ad(\psi_0(g))\circ\rho^\psi= \rho^\psi\circ\Ad(g)
\end{eqnarray*}
\end{enumerate}
\end{definition}

The category of  SHCP (denoted with $\shcps$)
is equivalent to the category of 
supergroups (denoted with $\sgrps$) as the next proposition states.
We refer the reader to \cite{cf} and \cite{vi2} for all of the details.
We
shall write $\shcps_\R$ or $\shcps_\C$ whenever it is
necessary to distinguish between real or complex SHCPs.

\begin{theorem} \label{eq-cat}
Define the functors
$$
\begin{array}{ccc}
{\mathcal{H}}:\, \sgrps & \to & \shcps \\
G & \to &({G_0}, \Lie(G) )\\
\phi & \to & (|\phi|, (d \phi)_e)
\\ \\
\cK:\, \shcps & \to & \sgrps \\
(G_0,\fg) & 
\to & (G_0, \Hom_{\cU(\fg_0)} \big( \cU(\fg), \cO_{G_0} \big)) \\
\psi=(\psi_0,\rho^\psi) & \to & f \mapsto \psi^*_0 \circ f \circ \rho_\psi
\end{array}
$$
where $G$ and $(G_0, \fg)$ are objects and
$\phi$, $\psi$ are morphisms of the corresponding categories
(in the definition of ${\mathcal{H}}$, 
$G_0$ is the ordinary group underlying $G$).
Then ${\mathcal{H}}$ and $\cK$ define an equivalence between the categories
of supergroups (differentiable or analytic) and super Harish-Chandra pairs
(differentiable or analytic).
\end{theorem}

We now want to give the definition of real form, see Def. \ref{realforms},
through the language of SHCP's.

\begin{definition} \label{shcp-realforms}

Let $(G_0, \fg)$ be a complex analytic SHCP. We say that
the pair $(r_0, \rho^r)$ is a \textit{real structure} on $(G_0, \fg)$ if
\begin{enumerate}
\item $r_0:G_0 \lra \overline{G_0}$ is a real structure on the ordinary
complex group $G_0$, with fixed points being a real
Lie group denoted by $G_0^{r}$. Notice that $r_0$ is an involutive automorphism
of the ordinary real Lie group underlying $G_0$.

\item $\rho^r: \fg \lra \fg$ is a $\C$-antilinear 
involutive Lie superalgebra morphism,
with its fixed points $\fg^{r}\simeq\Lie(G_0^r)$.

\item 
$(r_0, \rho^r)$ are compatible in the sense of Def. \ref{def:SHCP}, 
that is $(dr_0)_{1_{G_0}}=\rho^r_{\left.\right|_{\fg_0}}$
and $\rho^r$
intertwines the adjoint action. In other words $(r_0, \rho^r)$ is
an involutive automorphism of  $(G_0, \fg)$ as real Lie supergroup.
\end{enumerate}
 
Furthermore, we say that given a real structure $r=(r_0, \rho^r)$, $(G_0^r, \fg^r)$
is a \textit{real form} of $(G_0, \fg)$. 
\end{definition}

\begin{observation} 
If $G$ is a complex supergroup, with SHCP $(G_0, \fg)$,
given a real form  associated with 
a real structure $r$ in the sense of Def. \ref{realforms}, we have
that $(r_0, (dr)_{1_{G_0}})$ is a real form of $(G_0, \fg)$, in the sense of
Def. \ref{shcp-realforms}. Vice-versa, if we have a real form as in
Def. \ref{shcp-realforms}, by the 
equivalence of categories in Theorem \ref{eq-cat} we can associate
to $(r_0, \rho^r)$ a real structure as
in Def. \ref{realforms} and thus obtain a real form in
the sense of Def. \ref{realforms}.
\end{observation}

We conclude this section with the definition of representation of
a supergroup, which works in the three categories of supergroups
we have introduced: real differentiable, real analytic and complex
analytic.

\begin{definition}
Let $G=(|G|, \cO_G)$ be a supergroup 
and $V$ a finite dimensional super vector space. 
A \textit{representation} of $G$ in $V$ is a morphism of supergroups:
$$
\rho: G \lra {\mathrm{Aut}}(V)
$$
If we fix a basis for $V$, so that $V\cong \kk^{m|n}$ ($\kk=\R$ or $\C$),
we obtain a morphism of $G$ into
$\rGL(m|n)$. Hence $\rho(G(T))$ consists of certain matrices in $\rGL(m|n)(T)$
the invertible $m|n \times m|n$ matrices with coefficients in $\cO(T)$,
where $T$ is a supermanifold. 
The function:
$$
\begin{array}{cccc}
a_{ij}: & G(T) & \mapsto & \kk^{1|1}(T) \\
& g & \mapsto & \rho(g)_{ij}
\end{array}
$$
which associates to each $g \in G(T)$ the $(i,j)$ entry $\rho(g)_{ij}$ 
of the matrix $\rho(g)_{ij} \in \rGL(m|n)(T)$ is called a \textit{matrix
element} or equivalently a 
\textit{representative function}
of the representation $\rho$. $a_{ij}$ may be as well interpreted as
an element in $\cO(G)$, since by the Chart's Theorem we have the 
correspondence between the morphisms of $G \lra \kk^{1|1}$ and
the choice of a pair (that is the sum) of an even and an odd section
in $\cO(G)$.
\end{definition}

We shall also be looking at complex representations of a real
Lie supergroup $G$. This means that we look at morphisms of
a real Lie supergroup $G$ into the complex general linear supergroup
viewed as a real supergroup (of twice the dimension). 
Hence $a_{ij}$ in this case corresponds to
an element of $\cO(G) \otimes \C$ the complexification of $\cO(G)$.

\medskip

In Sec. \ref{shcp-rep-fns} we are going to revisit the notion
of matrix element of a supergroup associated with a given representation
in the language of SHCP's.

\section{The supergroup $S^{1|1}$} \label{S11}

We want to construct the real supergroup $S^{1|1}$, a supergroup
of dimension $1|1$ with reduced space $S^1$, as the unique 
compact 
real form of the supergroup $(\C^{1|1})^\times$ using the
language of SHCP's. We start with the description
of the \textit{skew-field} $\bD$ introduced  in \cite{dm}.
As a super vector space, $\bD \cong \C^{1|1}$, however its superalgebra
structure will endow $(\C^{1|1})^\times$ with a natural multiplication,
turning it into a supergroup,
as we shall presently see.

\begin{definition} We define
$\mathbb{D}_k$ as the noncommutative complex superalgebra
\begin{align*}
\mathbb{D}_k := \mathbb{C}[\theta_k], \quad \theta_k \text{ odd}, 
\quad \theta_{{k}}^2 = -k,\quad k \in \C
\end{align*}
\end{definition}

For any $k \neq 0$, $\mathbb{D}_k$ is a central simple 
superalgebra, as the reader will readily check, while $\mathbb{D}_0$ is a 
free commutative superalgebra on one odd variable. 
We will denote $\mathbb{D}_1$ simply by $\mathbb{D}$, 
and $\theta_1$ by $\theta$.

For $k \neq 0$, over the complex field, we have the isomorphism
$\mathbb{D}_k \cong \mathbb{D}$, $\theta_k \mapsto \pm \sqrt{k} \theta$, where 
we may choose any of the two square roots of $k$. 
The opposite superalgebra $\mathbb{D}^o_k$ of $\mathbb{D}_k$ is 
$\mathbb{D}_{-k}$.

\begin{remark} 
The $\bD_k$ form a family of superalgebras over $\mathbb{C}^\times$ by 
$k \mapsto \bD_k$ 
that we may study using deformation theory.
The family $\bD_k$ can be thought as a holomorphic deformation of superalgebras 
over $\mathbb{C}^\times$, with distinguished member $\bD_1$. 
Further, this deformation should be locally trivial, 
but globally nontrivial (i.e., not isomorphic to a product deformation) 
because of the nonexistence of a holomorphic square root of 
$z$ on $\mathbb{C}^\times$. 
We shall not pursue further this point in the present paper.
\end{remark}

Consider now the functor:
$$
\begin{array}{ccc}
F:\smflds^o & \lra &\sets \\
T & \to & [\mathcal{O}_T \otimes \mathbb{D}_k]^\times_0 
\end{array}
$$
where the index $\times$ denotes the units,
the definition on the morphisms being clear.

\smallskip\noindent
We have quite immediately the following proposition.

\begin{proposition}
The functor $F$ is the functor of points of the analytic Lie supergroup 
$(\C^{1|1})^\times:=(\C^\times, \cO_{\C^{1|1}}|_{\C^\times})$ with
group law:
$$
(w,\eta) \cdot (w',\eta')=(ww'+k \eta\eta',w\eta'+w'\eta), \qquad k \in \C
$$ 
\end{proposition}

We leave to the reader the easy check that the given operation is a group
law, with $(1,0)$ the unit and $(w^{-1}, -w^{-2}\eta)$ the inverse of an 
element $(w,\eta)$. Hence $(\C^{1|1})^\times$ admits a family of analytic 
supergroup structures parametrized by $k$, corresponding to the family of
superalgebra structures $\mathbb{D}^\times_k$ described above. We shall
denote the supermanifold $\G$ with the group structure depending on $k$
as $\G_k$.

\medskip
The algebra isomorphisms $\mathbb{D}_k \to \mathbb{D}$ 
induce supergroup isomorphisms $\G_k \to \G_1$ over
the complex field for $k \neq 0$; consequently the supergroups $\G_k$ are all
isomorphic as complex analytic supergroups, for $k \neq 0$.

\medskip
We now turn to the description of the supergroup $\G_k$ in terms of SHCP's.

\medskip
The Lie superalgebra $\fg^{1|1}_k$ of $\G_k$ is generated by the left invariant
vector fields
$$
C=w\partial_w+\eta \partial_\eta, \qquad Z=-k \eta \partial_w +w \partial_\eta,
\qquad k \in \C
$$
with brackets:
$$
[C,C]=[C,Z]=0, \qquad [Z,Z]=-2kC
$$
Hence the SHCP associated with $\G_k$ is $(\C^\times, \fg^{1|1}_k)$. Again,
all of these SHCP's are isomorphic (when $k \neq 0$),
as the reader can readily check.

\medskip
We now turn to the question of defining {\it real forms} of the 
supergroup $\G_k$, which correspond to $S^1$ on the
reduced part. We use first the SHCP's approach. 
According to Def. \ref{shcp-realforms} a real form of $\G_k$  
amounts to choosing a real form of the reduced group $\C^\times$.
We choose the real form of $\mathbb{C}^\times$ to be $S^1$, and a 
$\C$-antilinear involution
of $\fg^{1|1}_k$, which reduces to the suitable involution on the even
part of $\fg^{1|1}_k$, that is the involution corresponding to the 
choice of $S^1$ as real
form of $\C^{\times}$.
The most general form of such an involution of $\fg^{1|1}_k$ is:
\begin{equation} \label{lieinv}
\begin{array}{cccc}
\rho: & \fg^{1|1}_k & \lra & \fg^{1|1}_k \\
& C & \mapsto & -C \\
& Z & \mapsto & aZ
\end{array}
\end{equation}
Notice that
$C$ has to be mapped to $-C$ in order to have the corresponding group
$S^1$ on the reduced part, while $Z$ goes to a multiple of itself.
A small calculation shows that: $a=\pm i \overline{k}/|k|$. Choose
$a= i \overline{k}/|k|$ (the other case being the same).

\bigskip

The real form of $\fg^{1|1}_k$, consisting of those elements fixed by 
$\rho$ is generated (over $\R$) by:
$$
C'=iC, \qquad Z'=bZ
$$
A small calculation on the brackets shows that the coefficient $b$ must
be chosen such that 
$$
\frac{b}{\overline{b}}= i \frac{\overline{k}}{|k|}
$$
We have then a family of real forms of $\fg^{1|1}_k$ generated by $C'=iC$
and $Z'=bZ$ with brackets:
$$
[C',C']=[C',Z']=0, \qquad [Z',Z']=-2|b|^2|k|C'
$$
All of these superalgebras are isomophic over the reals, for $k \neq 0$. 

\medskip

We have proven the following proposition.

\begin{proposition} 
The Lie superalgebra $\fg_k^{1|1}$ admits up to isomorphism
a unique real form $\fg_{k,\R}^{1|1}$ with even part $\langle iC \rangle$,
described above. 
\end{proposition}

We can then define the real Lie supergroup  $S^{1|1}$
as the SHCP  $(S^1, \fg_{k,\R}^{1|1})$, where $S^1$ acts trivially on 
 $\fg_{k,\R}^{1|1}$.
Notice that, by the previous
proposition, we have that all of such real
supergroups are isomorphic. 

\smallskip\noindent
By its very definition $S^{1|1}$ is a real form of $\G$ and
it is compact since its underlying topological space is compact.

\section{A geometric approach to the supergroup $S^{1|1}$}
\label{geo-S11}

In our previous section we have
established all of the possible involutions
giving rise 
to the real forms of $\G_k$ with the language of SHCP's. We
now want to recover the same involutions at the supergroup level
using the functor of points notation, so as to make our calculations
more explicit.

\smallskip
We shall at first consider real structures given by the composition of a 
SUSY preserving holomorphic automorphism of $\G_k$, 
followed by a complex conjugation 
$\G_k \to (\overline{\mathbb{C}^{1|1}})^\times_k$ 
(refer to Def. \ref{realforms}). By
our previous discussion of SHCP's, we will then see that all real structures
giving us the real forms of $\G_k$ are of this form. This fact
is very remarkable, since the SUSY curves are in themselves very interesting
objects, extensively studied by Manin in \cite{ma1}. It is not so surprising
though, because of the tight connection between $\bD^\times$ and the 
SUSY structures (see for example \cite{fk} for a survey on basic 
facts of SUSY curves).

\medskip
Let us start by briefly recalling 
the notion of SUSY structure as in \cite{ma1}, by
Manin. A {\it SUSY-1 structure} (or SUSY structure for short)
on a $1|1$ complex supermanifold $X$ is a rank $0|1$ holomorphic 
distribution $\mathcal{D} \subseteq TX$ such that the Frobenius map
\begin{align*}
&\mathcal{D} \otimes \mathcal{D} \to TX/\mathcal{D}\\
&Y \otimes Z \mapsto [Y, Z] \text{ mod } \mathcal{D}
\end{align*}
is an isomorphism. 

The supergroup $\G_k$ carries a natural right-invariant 
SUSY-1 structure, defined by the vector field
\begin{align*}
Z_k = k \eta \partial_w + w \partial_\eta
\end{align*}
One may check that $Z_k, Z^2_k$ span the tangent space of 
$\G_k$ at $(1,0)$, hence since they are right-invariant, 
they span the tangent space 
of $\G_k$ at every point. Thus the span of $Z_k$ is a SUSY-1 structure.
We can define the SUSY-1 structure from the dual point of view by using differential one-forms. Let 
$\omega_k = w \, dw - k\eta \, d\eta$. 
One checks that $ker(\omega_k) =$ span$\{Z_k\}$ (see \cite{fk} for
more details).

\begin{proposition} 
Let us consider the morphisms of analytic
supermanifolds $P_{\pm}:\G_k \lra \G_k$ given by:
\begin{align*}
P_{\pm}(w, \eta) = (w^{-1}, \pm iw^{-2} \eta)
\end{align*}
where $(w,\eta)$ are global coordinates on $\C^{1|1}$.
Then $P_\pm$ are automorphisms of the supergroup $\G_k$ and furthermore
they are the unique SUSY-1 preserving endomorphisms of $\C^{1|1}$, 
that restrict to $w \mapsto w^{-1}$ on the reduced group $\mathbb{C}^\times$.
\end{proposition}

\begin{proof}
We first check that $P_\pm$ are automorphism of the supergroup $\G_k$:
\begin{align*}
P_{\pm}[(w, \eta) \cdot (w', \eta')] &= P_\pm(ww'+ k\eta \eta', w \eta'+ w'\eta)\\
&=((ww'+ k\eta \eta')^{-1}, \pm i(ww'+ k\eta \eta')^{-2}(w \eta'+ w'\eta))\\
&= ((ww')^{-1} - k(ww')^{-2}\eta \eta', \pm i(w^{-1}w'^{-2}\eta'+ w^{-2}w'^{-1} \eta))
\end{align*}
\begin{align*}
P_\pm(w, \eta) \cdot P_\pm(w', \eta') &= (w^{-1}, \pm iw^{-2} \eta) \cdot (w'^{-1}, \pm iw'^{-2} \eta')\\
&=(w^{-1}w'^{-1}-k(ww')^{-2} \eta \eta', \pm i(w^{-1} w'^{-2} \eta' + w^{-2}w'^{-1} \eta))
\end{align*}

Suppose now $F(w, \eta)$ is an endomorphism that restricts to 
$w \mapsto w^{-1}$ on $\mathbb{C}^\times$. 
Then $F(w, \eta) = (w^{-1}, g(w) \eta)$ for some function $g(w)$ of $w$. 
The SUSY-1 structure on $\C^{1|1}$ is determined by the differential form $w \, dw - k \eta \, d\eta$.

$F$ preserves the SUSY-1 structure if and only if 
$F^*(\omega) = h(w) \omega$ for some even invertible function 
$h$ (see \cite{fk}, Lemma 5.2). We have:
\begin{align*}
F^*(w, \eta) &= w^{-1} \, d(w^{-1}) - g \eta \, d(g \eta)\\
&= -w^{-3} \, dw - g^2 \eta \, d\eta
\end{align*}

The condition $F^*(\omega) = h(w) \omega$ is equivalent to $-w^{-3} = hw$, 
$h = g^2$. This is true if and only if $g^2 = -w^{-4}$, which in turn is true 
if and only if $g = \pm i w^{-2}$.
\end{proof}

\medskip

We want to define {\sl real forms} of the supergroups $\G_k$. 
The next proposition establishes all of the possible real structures on $\G_k$ which reduce to the usual complex conjugation on $\C^\times$ (i.e., the one induced by standard linear complex conjugation on $\mathbb{C}$).
We shall refer to them as \textit{complex conjugations} of 
$\G_k$. 

\begin{proposition}
Let $s_k:\G_k \lra (\overline{\mathbb{C}}^{1|1})^\times_k$ ($k \neq 0$) be 
a supergroup real structure, reducing to the usual complex 
conjugation in $\C^\times$, i.e. $|s_k|$ is the usual complex conjugation 
on $|(\mathbb{C}^{1|1})^\times|$. Then on $T$-points, $s_k$ is of the form:
$$
s_k(w, \eta)=
(\overline{w}, u \overline{\eta})
$$
where $u^2= \frac{\overline{k}} {k}$. In particular, $u$ 
is a complex number of modulus $1$.
\end{proposition}

\begin{proof}
Let us consider the action of $s_k$ on the functor of $T$-points of $(\mathbb{C}^{1|1})^\times_k$. Then since the restriction of $s_k$ to the underlying space is ordinary linear complex conjugation, we have $s_k(w, \eta) = (\overline{w}, u \overline{\eta})$ at the level of $T$-points, where $u$ is (the pullback to $T$ of) an invertible even function.

We have:
\begin{align*}
s_k(w, \eta) \; {\cdot} \; s_k(w', \eta') &= (\overline{w}, u \overline{\eta}) \; {\cdot} \; (\overline{w}', u \overline{\eta}')\\
&=(\overline{w}\overline{w}' + ku^2 \overline{\eta} \overline{\eta}', u(\overline{w} \overline{\eta}' + \overline{w}'\overline{\eta})\\
s_k[(w, \eta) \cdot (w', \eta')]&= s_k(ww'+ k \eta \eta', w \eta' + w' \eta)\\
&=(\overline{ww' + k \eta \eta'}, u(\overline{w \eta' + w' \eta}))\\
&=(\overline{w}\overline{w}' + \overline{k} \overline{\eta} \overline{\eta}', u(\overline{w} \overline{\eta}' + \overline{w}'\overline{\eta})
\end{align*}

We see then that $s_k$ is a supergroup morphism if and only if $u^2 = \overline{k}/k$, where this holds for any $T$-point. This in turn implies that $u^2 = \overline{k}/k$ identically as functions. A calculation with the chain rule shows that $u$ is constant. Taking the modulus of both sides of the equation $u^2 = \overline{k}/k$, we see that $u$ is a complex number of modulus $1$. It is readily checked that this implies $s_k$ is involutive and hence a real structure.
\end{proof}

\medskip
Consider the involutive isomorphism $\rho_k$
obtained by composing $P_{+}$ with $s_k$ (we choose $+$), 
(refer to Def. \ref{realforms}):
$$
\rho_k: (\C^{1|1})_{k}^\times  \lra  
\overline{(\C^{1|1})}_{k}^\times, \qquad
\rho_k(w,\eta)=(\wbar^{-1},   iu\wbar^{-2} \etabar)
$$

\medskip
We now define the functor $X:\smflds \lra \sets$ as
the $T$-points of the real
supermanifold $(\C^{1|1})_{k,\R}$ satisfying the
relations obtained through $\rho_k$: 
$$
X_k(T)=\{(w,\eta,\wbar,\etabar) \in (\C^{1|1})_{k,\R}^\times(T) \, | \, 
w=\wbar^{-1}, \, \eta= iu\etabar\wbar^{-2}\}
$$
(we are under the convention explained in Examples 
\ref{realform-ex11} and \ref{realform-fpts}).
We notice immediately that $X_k(T)$ is group valued, since $\rho_k$ is
a supergroup morphism by its very construction, but one can also
verify this directly with a simple calculation.

\smallskip
We now want to show that $X_k$ is the functor of points of a
real analytic supergroup, 
which is indeed, as we shall see, $S^{1|1}$, 
described in the previous section in a very different language.

\smallskip
Most immediately $X_k$ corresponds to the
superspace $(S^1, \cO_{X_k,\C})$, where $\cO_{X_k,\C}$ is the quotient of the
sheaf of $(\C^{1|1})_{k,\R}^\times$ by the relations
$w=\wbar^{-1}, \, \eta= iu\etabar\wbar^{-2}$. This sheaf corresponds
classically to the complex valued functions on the real analytic
group $S^{1}$ defined as the fixed
points in $\C$ by the involution $w \mapsto\wbar^{-1}$.
In order to show that $X_k$ is a supermanifold, 
we first need to consider a real form of the
sheaf $\cO_{X_k,\C}$ and then we need to find local coordinates
at each topological point $x\in S^1$; such local
coordinates at $x$ will exhibit explicitly the local
isomorphism $X_k|_U \cong \R^{1|1}$, $U$ a neighbourhood of $x$. 

We first write the real equations defining
$X_k$ in $\C^{1|1}_{\R} \cong \R^{2|2}$. For simplicity, we set $k=1$, hence $u=1$, in the following calculations; the general case differs only slightly from this one and the details are left to the reader.

\smallskip
If we set the global real coordinates
on $\R^{2|2}$ as:
$$
x=\frac{w+\wbar}{2}, \quad y=\frac{w-\wbar}{2i}, 
\quad \sigma=\frac{\eta+\etabar}{2}, \quad \zeta=\frac{\eta-\etabar}{2i}
$$

we obtain the three real relations:

$$
x^2+y^2=1, \quad \sigma(x^2-y^2)+2xy\zeta=\zeta, \quad
\zeta(x^2-y^2)-2xy\sigma=\sigma,
$$
which define the structure sheaf of the real analytic superspace $X := X_1$ as a quotient of the sheaf of $\R^{2|2}$. 

These three equations do not have linearly independent differentials at all topological points, so they do not cut out $S^{1|1}$ as a global complete intersection in $\R^{2|2}$. However, we can cover $S^1$ by two charts so that on the topological points satisfying $x^2+y^2=1$, that
is for $x=\cos(t)$ and $y=\sin(t)$, we obtain that each of the two equations
$$ 
\sigma(x^2-y^2)+2xy\zeta=\zeta, \quad
\zeta(x^2-y^2)-2xy\sigma=\sigma,
$$
is a multiple of the other.

 Such charts correspond to
the conditions $1-\sin(2t) \neq 0$ and $1 + \sin(2t) \neq 0$; 
we note that $1-\sin(2t)$ and $1+\sin(2t)$ never simultaneously vanish 
as functions of $t$. 
Thus, under these assumptions, we have respectively
$$
\zeta=\frac{\cos(2t)}{1-\sin(2t)}\sigma, \qquad 
\sigma=\frac{\cos(2t)}{1+\sin(2t)}\zeta,
$$
providing the two sets of local coordinates:
$$
\begin{array}{c}
(t,\sigma), \qquad \hbox{for} \, 1-\sin(2t) \neq 0 \, \hbox{and}\, t \in (0,\frac{\pi}{4}) \cup
(\frac{\pi}{4},\frac{5\pi}{4}) \cup (\frac{5\pi}{4},2\pi) 
\\ \\
(t,\zeta), \qquad \hbox{for} \, 1+\sin(2t) \neq 0 \, \hbox{and}\, t \in (0,\frac{3\pi}{4})\cup(\frac{3\pi}{4},\frac{7\pi}{4})\cup)\cup(\frac{7\pi}{4},2\pi)
\end{array}
$$
Note that these charts are real analytic. 
Hence we have proven the superspace $X$ is a real analytic supermanifold. 

\medskip
We can now state the main result of this section, relating this
geometrical picture with the SHCP construction in the previous
section.

\begin{proposition}
\begin{enumerate}
\item The fixed points of the involution $\rho_k$ define the functor of
points of a real analytic Lie supergroup $G$, corresponding to the
SHCP $S^{1|1}_k=(\C^\times, \fg_k^{1|1})$ defined in the previous section:
$$
S^{1|1}_k(T)=\{(w,\eta,\wbar,\etabar) \in (\C^{1|1})_{k,\R}^\times(T) \, | \, 
w=\wbar^{-1}, \, \eta= iu\etabar\wbar^{-2}\}
$$
and any two such are isomorphic for all values of $k \neq 0$.
\item Any real structure on $(\C^{1|1})_k^\times$ is obtained by
composing a SUSY preserving automorphism of 
$(\C^{1|1})^\times_k$ with a complex conjugation, hence it is one
of the $\rho_k$. 
\end{enumerate}
\end{proposition}

\begin{proof}
According to the discussion before the statement of the theorem,
the only thing that remains to be checked is the fact that the Lie superalgebra of $S^{1|1}_k$ is $\fg_k^{1|1}$.

The problem reduces to computing the differential of:
$$
\rho_k: (\C^{1|1})_k^\times  \lra  (\overline{\C^{1|1}})_k^\times \quad
\rho_k^*(\wbar, \etabar)=(w^{-1},   iu w^{-2} \eta)
$$

\noindent at the point $(1,0)$ in $(\mathbb{C}^{1|1})^\times$. Note here that we are not using the functor of points notation, but we are specifying the pullbacks of the global coordinates $\overline{w}, \overline{\eta}$ on $(\overline{\mathbb{C}^{1|1}})^\times_k$ under $\rho_k$.

The $\C$-linear map $(d\rho_k)_{(1,0)}$ is identified with the real structure on $T_{(1,0)}(\C^{1|1})^\times_k = \fg^{1|1}_k$. A simple calculation shows that
$$
(d\rho_k)_{(w, \eta)}= \begin{pmatrix} -w^{-2} & -2iu w^{-3} \eta \\
0 & iuw^{-2}
\end{pmatrix}
$$
Hence
$$
(d\rho_k)_{(1,0)}(C)=-\overline{C}, \qquad (d\rho_k)_{(1,0)}(Z)=iu\overline{Z}
$$
\noindent where $\overline{C}, \overline{Z}$ are the conjugate of the basis $C, Z$ of $\fg^{1|1}_k$. Now regarding the matrix of $(d\rho_k)_{(1,0)}$ as the matrix representing the corresponding $\mathbb{C}$-antilinear map, we see these are exactly the conditions we have in (\ref{lieinv}), which define $\fg_k^{1|1}$.
\end{proof}

From now on and for the rest of this note we shall then take $k=1$.

\smallskip\noindent
We end this section with a remark on the universal cover of $S^{1|1}$.

\begin{remark}
We want to show that the real additive supergroup
$\R^{1|1}$ with group law:
\begin{equation} \label{addgrplaw}
(t, \tau) \cdot (t', \tau') = (t + t' + \tau \tau', \tau + \tau').
\end{equation}
is the supergroup corresponding topologically
to the universal cover of $S^{1|1}$, that is, we have a surjective
morphism of supergroups $\R^{1|1} \lra S^{1|1}$, which
is a local diffeomorphism. 
Consider first the complex analytic
Lie supergroup $\C^{1|1}$ with same group law as (\ref{addgrplaw}) and
the super exponential map $\mathrm{Exp}: \C^{1|1} \to \G$, 
$\mathrm{Exp}(z, \zeta)$ $=$ $(e^z, e^z \zeta)$.
We are going to check that $\C^{1|1}$ is the universal cover of  $\G$,
that is, its topological space is simply connected and we define
a surjective morphism from  $\C^{1|1}$ to $\G$ which is a local diffeomorphism.
Let us define:
\begin{align*}
p(t, \tau) = (e^{it}, e^{\pi i/4}e^{it}\tau)
\end{align*}

We show that $p$ maps $\R^{1|1}$ into $S^{1|1}$; it is enough to verify 
that $p$ is invariant under $s=s_1$. 
$$
\begin{array}{rl}
s \circ p(t, \tau) &= \phi((e^{it}, e^{2\pi i/4}e^{it}\tau)) = 
((\overline{e^{it}})^{-1}, i \overline{e^{-2it}} 
\overline{e^{\pi i/4} e^{it} \tau}) =
(e^{it}, ie^{-\pi i/4}e^{it} \tau)\\
&=(e^{it}, e^{\pi i/4}e^{it} \tau) =p(t, \tau).
\end{array}
$$
Hence $p$ is a homomorphism of the supergroup $\R^{1|1}$ into the supergroup 
$S^{1|1}$. The map of reduced spaces is surjective, with discrete 
kernel, hence a covering map of Lie supergroups.
\end{remark}

\section{The representations of  $S^{1|1}$} \label{S11-reps}

In the language of SHCP's (refer to Sec. \ref{prelim}) we can
identify the supergroup $S^{1|1}$ with the SHCP 
$(S^1, \fg^{1|1}_\R)$, hence a representation of $S^{1|1}$ consists of a
pair: a representation of $S^1$ together with a representation of $\fg^{1|1}_\R$
satisfying some compatibility conditions (see \cite{ccf} Ch. 7
for the details).

\medskip\noindent
We  describe a large family of complex semisimple
representations of $S^{1|1}$  and we show that any $S^{1|1}$-representation, whose weights are all nonzero, is a direct sum of members of our family. We also calculate the matrix elements of the members of this family.

We denote with $(\pi_0,\rho^{\pi_0})$ the trivial representation of $S^{1|1}$ on $\C$. It is  the representation defined by
\begin{align*}
\pi_0(t)=\id,\qquad
\rho^{\pi_0}=0
\end{align*}

We then define, for each $m\neq 0$, a key class of $1|1$-dimensional representations of $S^{1|1}$ as follows. The reduced group $S^1$ acts with integer weight $m$ on $V$: $t \cdot v := t^mv$, for all $v \in V$, dim$(V)=1|1$.  It remains to define the action of $\fg^{1|1}_\R$. The action of $C$ is obtained by differentiating the action of $S^1$: $C \cdot v = mv$. We require that there is a homogeneous basis $v_0, v_1$ such that in this basis, $Z$ is given by the matrix:

$$
\begin{pmatrix} 0 & \sqrt{-m} \\
\sqrt{-m} & 0 \end{pmatrix}.
$$

It is easily checked that the commutation relations for $\fg^{1|1}_\R$ are satisfied, so that this defines a representation of $\fg^{1|1}_\R$, that we denote with:
\[
(\pi_m , \rho^{\pi_m})\qquad (m\neq 0)
\]
We shall call a representation defined as above a {\it super weight space of weight $m$}. For each $m$ there are two super weight spaces of weight $m$, depending on the choice of $\sqrt{-m}$, but both choices yield isomorphic representations since the matrices:

\begin{align*}
\left( \begin{array}{c|c} 0 & \sqrt{-m} \\
\hline\
\sqrt{-m} & 0 \end{array} \right), 
\left( \begin{array}{c|c} 0 & -\sqrt{-m} \\
\hline
-\sqrt{-m} & 0 \end{array} \right)
\end{align*}

\noindent are seen to be conjugate by the matrix

\begin{align*}
\left( \begin{array}{r|r} i & 0 \\
\hline
0 & -i \end{array} \right).
\end{align*}

\noindent which certainly commutes with the action of the reduced group $S^1$, thus with that of $C$.

It is easily seen that a super weight space of weight $m$, $m \neq 0$, is an irreducible representation of $S^{1|1}$. The next theorem shows that, under appropriate finiteness assumptions, all $S^{1|1}$ representations (with one key exception) are obtained as direct sums of super weight spaces.

\begin{theorem} \label{s11-reps}
Let $V$ be a complex linear representation of the 
super Lie group $S^{1|1}$. Suppose that the representation of
the reduced group $S^1$ on $V$ contains no 
trivial subrepresentations, and that for each $m \in \mathbb{Z}$, the super vector space $\{v \in V: t \cdot v = t^mv \text{ for all } t \in S^1\}$ is finite-dimensional. Then $V$
is isomorphic to a direct sum of super weight spaces.
In particular, if $V$ is irreducible and the weight of $S^1$ on $V$ is nonzero, then $V$ is a super weight space.
\end{theorem}

\begin{proof} Let $(\pi, \rho^\pi)$ be a finite dimensional representation
of the SHCP $(S^1, \fg^{1|1}_{\R})$ in $V$.
It is well known that any complex representation of $S^1$ is a direct sum of weight spaces, 
hence, we have a direct sum decomposition of
ungraded vector spaces $V = \oplus V_m$, where $S^1$ acts on $V_m$ by $t \cdot v = t^mv$.

By the definition of a SHCP representation, we have that
$$
\rho^\pi(g \cdot X) = \pi(g)\rho^\pi(X) \pi(g^{-1})
$$
for any $X \in \fg^{1|1}_{\R}$, $g \in S^1$. 
Since the adjoint action of $S^1$ on $\fg^{1|1}_{\R}$
is trivial, the action of $\rho^\pi$ commutes with 
the action of $S^1$. Therefore the weight spaces
$V_m$ are $\rho^\pi$-invariant, and are thus $\fg^{1|1}_{\R}$
subrepresentations of $V$. So $V = \oplus V_m$ as
representations of $\fg^{1|1}_{\R}$.
We may therefore assume from now on that $V = V_m$ for
some $m$. By hypothesis $m$ is nonzero, and $V_m$ finite-dimensional.
We must analyze the endomorphism $\rho^\pi$. Since $C$ and $Z$ generate 
$\fg^{1|1}_{\R}$, it is enough to determine the action of 
$\rho^\pi(C)$ and $\rho^\pi(Z)$. Since $C$ is just the
infinitesimal action of $S^1$, we have $\rho^\pi(C) = mI$. Note that we have:
$$
2\rho^\pi(Z)^2 = [\rho^\pi(Z), \rho^\pi(Z)]= \rho^\pi([Z,Z])= -\rho^\pi(2C)
= -2mI
$$
Hence $\rho^\pi(Z)^2 = -mI$. By the finite-dimensionality hypothesis, this implies $\rho^\pi(Z)$ is diagonalizable, considered as an endomorphism of the ungraded vector space $V$. 
The eigenvalues of $\rho^\pi(Z)$ are the square roots of $-m$. We choose one particular square root and denote it by $\sqrt{-m}$. Then $V$ splits into eigenspaces for $\pm \sqrt{-m}$, each of which is invariant under $Z$, hence under $C$ and thus under the SHCP of $S^{1|1}$. Hence we may further assume that $V$ is the $\sqrt{-m}$ eigenspace; the argument for the $-\sqrt{-m}$-eigenspace will be the same. 

From now on, we will abuse notation and denote the endomorphism $\rho^\pi(Z)$ by $Z$. Let $w$ be any eigenvector of $Z$, $w = w_0 + w_1$ its homogeneous decomposition. Now
$$
Z(w_0) + Z(w_1) = \sqrt{-m} w_0 +\sqrt{-m}w_1.
$$\

Since $Z$ is odd, $Z(w_0) =\sqrt{-m}w_1$, $Z(w_1) =\sqrt{-m}w_0$. 
It follows that both the even and odd components of $w$ are nonzero, 
for if either $w_0$ or $w_1$ were zero then $w$ would be zero. 
Applying this reasoning to a (not necessarily homogeneous) basis of eigenvectors of $V$ implies that there exists a homogeneous basis of $V$, $v_1$, $\dots$, $v_n$, $\nu_1$, $\dots$, $\nu_n$ such that $Z(v_i) = 
\sqrt{-m} \nu_i, Z(\nu_i) = \sqrt{-m} v_i$. Thus each pair $v_i$, $\nu_i$ spans a super weight space of weight 
$m$ and $V$ is a direct sum of super weight spaces of weight $m$, as desired. It follows immediately that if $V$ is irreducible with nonzero $S^1$-weight, $V$ is a super weight space.
\end{proof}

\begin{remark} \label{rem::reprred}We now show that the assumption that $S^1$ acts with nonzero weight is essential for reducibility, by producing an $S^{1|1}$-representation such that the reduced group $S^1$ acts with weight $0$, which has a nontrivial subrepresentation, but is not a direct sum of irreducibles.

\smallskip\noindent
We will define such a representation on $\mathbb{C}^{1|1}$ as follows. We take the action of $S^1$ to be the trivial action, so $C$ acts by zero. We fix a homogeneous basis $u, w$ of $\mathbb{C}^{1|1}$, and define an action of $Z$ by:
$$
Z \cdot u = w, \qquad
Z \cdot w = 0.
$$
One checks that these actions of $C$ and $Z$ satisfy the commutation relations for $\fg^{1|1}_\R$, so that we have indeed defined an action of $\fg^{1|1}_\R$ on $\C^{1|1}$ and hence an action of the SHCP $(S^1, \fg^{1|1}_\R)$. 
We denote this representation with
\[
(\pi_-,\rho^{\pi_-})
\]
and we notice that it is not completely reducible, since the span of $w$ is a nontrivial $S^{1|1}$-invariant subspace which does not admit an invariant complement.
\noindent
We have not formulated a notion of reductivity for super Lie groups, but the existence of representations of $S^{1|1}$ that are not completely reducible means that merely carrying over the ordinary definition will not work.

\end{remark}

Next we want to write explicitly the representations described
in \ref{s11-reps} and compute their matrix elements. We are also going to
realize $S^{1|1}$ as a real subgroup of the 
special unitary supergroup $\rSU(1|1)$.

Let us consider the real Lie superalgebra:
$$
\begin{array}{rl}
\rsu(1|1)&=\left\{
\begin{pmatrix} ix & z \\ -i\zbar & ix \end{pmatrix}\right\} \\ \\
&=\Span\left\{ \, -iI=\begin{pmatrix} -i & 0 \\ 0& -i \end{pmatrix}, \,
U=\begin{pmatrix} 0 & 1 \\ -i& 0 \end{pmatrix}, \,
V=\begin{pmatrix} 0 & i \\ -1& 0 \end{pmatrix} \,
\right\}
\end{array}
$$
described in detail at pg 111 in \cite{vsv2}.

\medskip
It is possible to construct the real Lie supergroup corresponding
to this Lie superalgebra:
$$
\rSU(1|1)=\left\{
\begin{pmatrix} a & \beta \\ -i\betabar a^2 & \abar^{-1} \end{pmatrix}
\, | \, a \abar (1+i \beta \betabar)=1 \right\}
$$
(notice that the relation $a \abar (1+i \beta \betabar)=1$ corresponds
to setting the berezinian equal to $1$ after some calculation).
This real supergroup has dimension $1|2$.

If we impose $\betabar=-\beta \abar^{2}$, we obtain
the following subgroup:
$$
G= \left\{
\begin{pmatrix} a & \beta \\ -i\beta & a \end{pmatrix}
\, | \, a \abar=1 \right\}
$$ 
whose multiplication is precisely the multiplication in $S^{1|1}$.
The Lie superalgebra of $G$ is:
$$
\begin{array}{rl}
\Lie(G)&= \Span_\R \left\{
iI=\begin{pmatrix} i & 0 \\ 0 & i \end{pmatrix}, 
U=\begin{pmatrix} 0 & 1 \\ -i & 0 \end{pmatrix} \right\} 
\subset  \\
\rsu(1|1)&= \Span_\R \left\{
iI=\begin{pmatrix} i & 0 \\ 0 & i \end{pmatrix}, 
U=\begin{pmatrix} 0 & 1 \\ -i & 0 \end{pmatrix},
V=\begin{pmatrix} 0 & i \\ -1 & 0 \end{pmatrix} \right\} 
\end{array}
$$
We now want to compute the matrix elements for the action
on $S^{1|1}$ on an irreducible space.

\medskip

We start with what we know on the Lie algebra elements, namely
that irreducible complex representations are of dimension
$1|1$ with basis $v_\pm$, such representations $\rho_m$ are parametrized by
the integer $m$ and $iI$ acts as $miI$, while $U$ as:
$$
\rho_m(iI)v_\pm = im v_\pm, \qquad \rho_m(U)v_\pm= \sqrt{-m}v_{\mp}
$$

To compute the matrix elements we need to compute the
exponential of such action:
$$
exp(\rho_m(iI)t+\rho_m(U)\theta)=e^{\rho_m(iI)t}(I+\rho_m(U)\theta)=
\begin{pmatrix} e^{imt} & \sqrt{-m}e^{imt}\theta 
\\ -i \sqrt{-m}e^{imt}\theta &  e^{imt}
 \end{pmatrix}
$$

The entries of this matrix are the super coefficients of the representation. Notice that the diagonal entries of the matrix are even, while the off-diagonal entries are odd sections. Hence, it is evident that matrix elements of this form are not enough to give all ordinary representative functions since the trivial characters are missing. In order to obtain the trivial odd coefficient we need to add the non-semisimple representation $(\pi_-,\rho^{\pi_-})$ defined in Remark \ref{rem::reprred}.


Hence we have proven
directly in this special case the Peter Weyl theorem.

\begin{theorem} {\bf The super Peter-Weyl theorem for $S^{1|1}$}. 
The complex linear span of the matrix coefficients of the representations $\{(\pi_m,\rho^{\pi_m})\}_{m\in \Z}$ and $(\pi_-,\rho^{\pi_-})$ is dense in $\cO(S^{1|1})\otimes \C$.
\end{theorem}

We will give another proof of this result in the next section.

\section{SHCP's approach} \label{shcp-rep-fns}

We now want to show an alternative  proof of the Peter-Weyl theorem for $S^{1|1}$ through the language of the SHCP's. We will briefly 
discuss representation coefficients in general.

\medskip
Let $G$ be a compact real Lie supergroup. The associated super Harish-Chandra 
pair is given by $(G_0,\fg)$, where $G_0$ is the reduced Lie group and 
$\fg$ the super Lie algebra of $G$. We  recall 
(see the discussion in Section \ref{shcp-sec}) that the sheaf of 
$G$ is canonically isomorphic to
\begin{equation}
\label{eq:sheafSHCP}
\cO(G)  =\Hom_{\cU(\fg_0)} (\cU(\fg), \cO(G_0))
\end{equation}
 We are interested in the finite
dimensional complex representations of $G$. In the following if $V$ denotes a finite dimensional, complex vector space, $V^\ast$ denotes the corresponding dual and $\pair{\cdot}{\cdot}$ the pairing between $V$ and $V^\ast$. With $V_\R$ we denote the vector space $V$ viewed as a real vector space.

 We are after matrix elements also called representative functions.
Let us consider the definition in the ordinary setting.

\begin{definition}
Let $G_0$ be a Lie group.
If $\pi: G_0 \lra \rGL(V)$ is a representation, 
 $\omega \in V^*$ and $v \in V$ we define  the associated
\textit{matrix element}:
$$
c_{\omega,v} (g)=\langle \omega, \pi(g) v \rangle
$$
\end{definition}

\medskip  

For the subsequent discussion we remark that, denoting with $a\colon G_0\times V \to V$ the  linear action associated with the representation $\pi$ through
\begin{align}
\label{eq:classcoeff}
a(g,v)=\pi(g)v\,,
\end{align}
 a matrix element can also be written as
\[
c_{\omega,v} (g)=\langle a^\ast(\omega)(g),  v \rangle
\]
where $a^*: V^* \lra \cO(G_0) \otimes V^*$.

\medskip

Notice that if $V$ is a complex vector space
\(
c_{\omega,v} \in \cO(G_0)\otimes \C
\).

If $\{e_i\}_{i=1}^{\dim V}$ is a basis of $V$, and  $\{e_i^*\}_{i=1}^{\dim V}$ is the corresponding dual basis, then 
$$
c_{ij}(g)=\langle e_i^*, \pi(g) e_j \rangle 
$$
is the $(i,j)$ entry of the matrix representing $\pi(g) \in \rGL(V)$ in the basis  $\{e_i\}_{i=1}^{\dim V}$. In particular
a matrix element is an element in $\cO(G_0)\otimes \C$ the complexification of the algebra of global
sections on $G_0$.

\smallskip\noindent

Notice that the very definition of the pull-back of a morphism is:
\[
\pi^\ast \colon C^\infty(\rGL(V) _\R)  \to C^\infty (G_0)
\]
Hence the matrix element $c_{\omega,v}$ is obtained by applying 
\(
\pi^\ast\otimes \id_\C
\)
to the complex function on $G_0$ 
\begin{align}
\label{eq::functonG}
\rGL(V)\to \C \quad A\mapsto \langle \omega, A v \rangle
\end{align}

\medskip

We now want to approach the matrix elements via the
SHCP.  We start by introducing the notions  of
representation and action in supergeometry.

\smallskip\noindent
A morphism $\pi: G \lra \rGL(V)$ of supergroups, corresponds to
a complex linear action $a: G \times V \lra V$. In the following we make explicit the latter notion.
A complex linear action of the SLG G on the complex vector space $V$ is a super  algebra map
\[
a^\ast\colon \cO(V_\R) \to \cO(G)\otimes \cO(V_\R)
\]
obeying the usual commutative diagrams for an action, and  such that, by the linearity requirement,
\begin{align}
\label{eq::pullact}
a^*:{V}^\ast \lra \cO(G) \otimes {V}^\ast
\end{align}
where, with abuse of notation, we denote with $a^\ast$ the map $a^\ast\otimes \id_C$.\\

The previous discussion justifies the  next definition.

\begin{definition}
We say that $G$ \textit{acts linearly} on the complex linear vector space $V$ if a complex linear morphism 
\[
a^*:{V}^\ast \lra \cO(G) \otimes {V}^\ast
\]
is given, such that
\begin{align*}
(\mu^\ast\otimes \id_{V^\ast} )a^\ast= (\id_{\cO_G}\otimes a^\ast)a^\ast \qquad (\mathrm{ev}_e\otimes \id_{V^\ast})a^\ast=\id_{V^\ast}
\end{align*}
where $\mu\colon G\times G \to G$ is the group multiplication and $\mathrm{ev}_e$ is the evalutation at the identity of $G$ (i.e. the pull-back of the embedding $\{e\} \hookrightarrow G$).
\end{definition}

\smallskip\noindent

In SHCP theory, we have that (see \eqref{eq:sheafSHCP})
\begin{align*}
\cO(G)\otimes \C &=\Hom_{\cU(\fg_0)} (\cU(\fg), \cO(G_0)\otimes\C)
\end{align*}

\begin{definition}
Given $\omega \in V^*$ and $v \in V$, 
we define $c_{\omega,v} \in  \cO(G)\otimes\C$ the 
\textit{matrix element} associated with $\omega$ and $v$
as follows (compare with \eqref{eq:classcoeff}):
$$
c_{\omega,v}(X)(g)=\langle a^*(\omega)(X)(g), v 
\rangle\quad \forall X\in \cU(\fg)\,,\, g\in G_0\,.
$$
\end{definition}

Notice that $a^*(\omega) \in \cO(G) \otimes V^*$ and the
arguments $X \in \cU(\fg)$ and $g \in G_0$ have to be thought as relative to the
first component $\cO(G)$.

\smallskip\noindent
In SHCP theory we have that the action $a$  corresponds
to the pair (see, \cite{dm}):
\begin{align}
\label{eq::actmn}
\aubar:G_0 \times V \lra V, && \rho:\fg \lra
\underline{\End}(V^*)^{op} \subset  Vec(V)^{op} \\
\aubar = a \circ (i \times \id_V) &&
\rho(X)=(X \otimes 1) \circ a^*  \notag
\end{align}
where $i: G_0 \lra G $ is the canonical embedding. We stress that here  $V$ is viewed as a supermanifold. Moreover, we have the ``reconstruction'' formula:
\begin{align}
\label{eq::reconform}
a^*(\omega)(X)=(1_{\cO(G_0)} \otimes \rho(X)) \aubar^*(\omega)
\end{align}
From the action $(\aubar, \rho)$ we define a representation $(\tilde{\pi}, \rho^\pi)$ of the SHCP $(G_0,\fg)$ on $(\rGL(V_0) \times \rGL(V_1),\underline{\End}(V) )$ by
\begin{align}
\label{eq::linact}
\tilde{\pi}(g)^\ast = (ev_g \otimes 1)\,\aubar^\ast\qquad \rho^\pi(X)=\rho(X)^\ast
\end{align}

Let us stress once again the difference between \eqref{eq::actmn} and \eqref{eq::linact}. In the former $V$ is considered as a supermanifold, while in the latter as a super vector space.

\medskip\noindent

Now we go about an important step.

\begin{lemma}
Let the notation be as above. Then we have:
$$
c_{\omega,v}(X)(g)=\langle \, (ev_g \otimes \rho(X) ) \aubar^*(\omega), v \rangle=
\langle \, \omega, \tilde{\pi}(g)\rho^\pi(X)v \rangle.
$$
\end{lemma}

\begin{proof}
As already noticed, in the super setting it is natural to generalize 
\eqref{eq:classcoeff} defining:
\[
c_{\omega,v}(X)(g)=\pair{a^\ast(\omega)(X)(g)}{v}
\]
Using \eqref{eq::reconform}, we have 
$$
c_{\omega,v}(X)(g)=\langle \, (ev_g \otimes \rho(X) ) \aubar^*(\omega)
\, , \, v \, \rangle
$$

Hence, using \eqref{eq::linact},  we have that:
$$
\begin{array}{rl}
c_{\omega,v}(X)(g) &= \langle a^*(\omega)(X)(g), v \rangle=
\langle (1_{\cO(G_0)} \otimes \rho(X)) \aubar^*(\omega)(g), v \rangle=\\ \\
&=\langle \, (ev_g \otimes \rho(X) ) \aubar^*(\omega), v \rangle=
\langle \, \omega, \tilde{\pi}(g)\rho^\pi(X)v \rangle.
\end{array}
$$
which gives our claim. 
\end{proof}

We want to prove the following theorem, which is the
Peter-Weyl theorem for $S^{1|1}$ in the language of SHCP's.

\begin{proposition}
The complex linear span of the matrix coefficients of the representations $\{(\pi_m,\rho^{\pi_m})\}_{m\in \Z}$, and $(\pi_-,\rho^{\pi_-})$ is dense in 
$\cO(S^{1|1})\otimes \C$.
\end{proposition}

\begin{proof}
We need the following facts from the previous section.  
Let $C,Z$ denote
the basis of $\Lie(S^{1|1})$, defined in Section \ref{S11}:
\begin{align*}
C=w \partial_w+\eta \partial_\eta,\quad
Z=-\eta \partial_w+w \partial_\eta
\end{align*}
We use the results from Section \ref{S11-reps}.
If $m\neq 0$ there is an irreducible representation $(\pi_m,\rho^{\pi_m})$ of $S^{1|1}$ on $\C^{1|1}$. Let $v,\nu$ be a homogeneous basis of $\C^{1|1}$ with $|v|=0$, and $|\nu|=1$.  Define
\begin{align*}
\pi_m(t)v= \chi_m(t) v, &\quad \pi_m(t)\nu= \chi_m(t) \nu\\
\rho^{\pi_m}(Z) v =\sqrt{m}\nu, & \quad \rho^{\pi_m}(Z) \nu =\sqrt{m}v
\end{align*}
where $\chi_m(t)=t^m$.

We have, for each $m\in \Z$
\begin{align*}
c_{m;v^*,v}(1)(t)&=\pair{v^\ast}{\pi_m(t)v} = \chi_m(t)\\ \\
c_{m;v^*,\nu}(Z)(t)&=\pair{v^\ast}{\pi_m(t)\rho^{\pi_m}(Z)\nu}  = \sqrt{m}\chi_m(t)
\end{align*}
For the representation $(\pi_-,\rho^{\pi_-})$ defined in Remark 
\ref{rem::reprred}, we have
\begin{align*}
c_{-;v^*,v}(1)(t)=\pair{v^\ast}{\pi_-(t)v}  = 1\\
c_{-;v^*,\nu}(1)(t)=\pair{v^\ast}{\pi_-(t)\rho^{\pi_-}(Z)\nu}  = 1
\end{align*}
\medskip

We have to prove that for each $\phi\in \cO(S^{1|1})\otimes\C$ and for each $\epsilon>0$, and for each $X\in \Uu(\Lie(S^{1|1}))$, there exist irreducible representations $(\pi_{m_1},\rho^{\pi_{m_1}}),\ldots, (\pi_{m_n},\rho^{\pi_{m_n}})$ and vectors $\{(\omega_{m_1},v_{m_1}),\ldots, (\omega_{m_n},v_{m_n}) \}$ such that 
\[
\sup_{z\in S^1}\left|
\left[D^L_X(
\phi-\sum_{i=1}^n c_{m_i;\omega_i, v_i}
)\right]^{\widetilde{}}(z)
\right|<\epsilon
\]
Notice that in this expression we allow $m_i=-$.

  We use the identification:
\begin{align*}
\Hom_{\Uu(\Lie(S^{1}))}\left(
\Uu((\Lie(S^{1|1}), \cinfty(S^1)
\right)\otimes \C & \to \cinfty(S^1)\left[
1,Z^\ast
\right]\\
\phi & \mapsto \phi_0+\phi_1 Z^\ast
\end{align*}
where
\begin{align*}
\phi_0 = \phi(1) \quad \phi_1= \phi(Z)
\end{align*}
By PBW theorem (see \cite{vsv1}), we have to consider two cases: 
$X=C^n Z$, and $X=C^n$. We only consider the former, the latter being simpler.
\begin{align*}
\sup_{z\in S^1}\left|
\left[D^L_{C^n Z}(
\phi - \sum_{i=1}^n c_{\omega_i, v_i}
)\right]^{\widetilde{}}(z)
\right| & = \sup_{z\in S^1}\left|
\left[
\phi(C^n Z) - \sum_{i=1}^n c_{\omega_i, v_i}(C^n Z)
\right](z)
\right| \\
& = \sup_{z\in S^1}\left|
D^L_{C^n}\left(
\phi( Z) - \sum_{i=1}^n c_{\omega_i, v_i}( Z)
\right)(z)
\right| 
\end{align*}
Let $\omega_i= c_i \nu^\ast$ and $v_i=v$. By the previous calculation, we have:
\begin{align*}
& = \sup_{z\in S^1}\left|
 D^L_{C^n}\left(\phi_1
 - \sum_{i=1}^n c_i \pair{\nu^\ast}{ \pi_i(\cdot)\nu}\right)(z)
\right| \\
& =  \sup_{z\in S^1}\left|
 D^L_{C^n}\left(\phi_1
 - \sum_{i=1}^n c_i \chi_i(\cdot)\right)(z)
\right| 
\end{align*}
The results now follows from the ordinary one (see \cite{btd}).
\end{proof}

\end{document}